\documentclass[12pt]{article}

\usepackage{amsmath, amssymb, latexsym, amsthm}
\usepackage{fullpage}
\usepackage{color}
\usepackage{tikz}
\usepackage{graphicx}
\usepackage{stmaryrd,amsbsy,enumerate}
\usepackage{hyperref}
\usepackage{mathscinet}
\usetikzlibrary{calc}
\usepackage{cancel}
\usepackage{xspace}
\usepackage[mathlines]{lineno}
\usepackage{url}
\usepackage{longtable}
\usepackage[obeyFinal,colorinlistoftodos,textsize=tiny]{todonotes}

\usepackage{chngcntr}
 
\counterwithout{figure}{section}

\usepackage{mathtools,amsmath}
\usepackage{array}

\usepackage{tikz}

\if10     % 11 = on,  10 = off
\usepackage[mathlines]{lineno}
\newcommand*\patchAmsMathEnvironmentForLineno[1]{%
  \expandafter\let\csname old#1\expandafter\endcsname\csname #1\endcsname
  \expandafter\let\csname oldend#1\expandafter\endcsname\csname end#1\endcsname
  \renewenvironment{#1}%
     {\linenomath\csname old#1\endcsname}%
     {\csname oldend#1\endcsname\endlinenomath}}% 
\newcommand*\patchBothAmsMathEnvironmentsForLineno[1]{%
  \patchAmsMathEnvironmentForLineno{#1}%
  \patchAmsMathEnvironmentForLineno{#1*}}%
\AtBeginDocument{%
\patchBothAmsMathEnvironmentsForLineno{equation}%
\patchBothAmsMathEnvironmentsForLineno{align}%
\patchBothAmsMathEnvironmentsForLineno{flalign}%
\patchBothAmsMathEnvironmentsForLineno{alignat}%
\patchBothAmsMathEnvironmentsForLineno{gather}%
\patchBothAmsMathEnvironmentsForLineno{multline}%
}
\linenumbers
\fi

\def\Cr{\text{\rm cr}}
\def\rcr{\overline{\text{\rm cr}}}
\def\floor#1{\lfloor{#1}\rfloor}
\def\bigfloor#1{\bigl\lfloor{#1}\bigr\rfloor}
\def\Bigfloor#1{\Bigl\lfloor{#1}\Bigr\rfloor}
\def\dd{\mathcal D}
\def\rr{\mathcal R}
\def\ee{\mathcal E}
\def\hh{\mathcal H}
\def\Nat{\mathbb{N}}
\def\cc{\mathcal C}
\def\mm{\mathcal M}

\newcommand{\cA}{\mathcal{A}}
\newcommand{\cF}{\mathcal{\rr}}
\newcommand{\cK}{\mathcal{K}}

\newcommand{\Real}{\mathbb{R}}

\DeclareMathOperator{\Hom}{Hom}

\newtheorem{theorem}{Theorem} %Defines \begin{theorem} to write "Theorem"
\newtheorem{theorem*}{Theorem} %Defines \begin{theorem} to write "Theorem"
\theoremstyle{definition}

\theoremstyle{remark}

\newcommand{\oururl}{\url{http://lidicky.name/pub/hill/}}

\newcommand{\vc}[1]{\ensuremath{\vcenter{\hbox{#1}}}}

\newcommand*\kf{\scalebox{0.2}{\begin{tikzpicture}
\filldraw (0,0) circle (1.mm);
\filldraw (0,1) circle (1.mm);
\filldraw(-0.87,-0.5) circle (1.mm);
\filldraw(0.87,-0.5) circle (1.mm);
\draw(0,0) -- (0,1);
\draw(0,0) -- (-0.87,-0.5);
\draw(0,0) -- (0.87,-0.5);
\draw(-0.87,-0.5)--(0.87,-0.5);
\draw(0,1) -- (0.87,-0.5);
\draw(0,1) -- (-0.87,-0.5);
\end{tikzpicture}}}

\def\yu{{N_4}}

\begin{document}

%\linenumbers  %% This enables line numbers in the whole document.

\title{Closing in on Hill's conjecture}
\author{
J\'ozsef Balogh \thanks{Department of Mathematical Sciences,
University of Illinois at Urbana-Champaign, Urbana, Illinois 61801, USA, {\tt
jobal@math.uiuc.edu}. Research is partially supported by NSF Grant DMS-1500121  and the Langan Scholar Fund (UIUC).}  \and
Bernard Lidick\'{y} \thanks{Department of Mathematics, Iowa State University. Ames, IA, USA. E-mail: {\tt lidicky@iastate.edu}. Research of this author is supported in part by NSF grant DMS-1600390.}
\and
Gelasio Salazar \thanks{Instituto de F\'\i sica, Universidad Aut\'onoma de San Luis Potos\'{\i}. San Luis Potos\'{\i}, Mexico. E-mail: {\tt gsalazar@ifisica.uaslp.mx}. Research of this author is supported by Conacyt grant 222667.} }

\maketitle

\begin{abstract}
Borrowing L\'aszl\'o Sz\'ekely's lively expression, we show that Hill's conjecture is ``asymptotically at least $98.5\%$ true''. This long-standing conjecture states that the crossing number $\Cr(K_n)$ of the complete graph $K_n$ is $H(n) := \frac{1}{4}\floor{\frac{n}{2}}\floor{\frac{n-1}{2}}\floor{\frac{n-2}{2}}\floor{\frac{n-3}{2}}$, for all $n\ge 3$. This has been verified only for $n\le 12$. Using the flag algebra framework, Norin and Zwols obtained the best known asymptotic lower bound for the crossing number of complete bipartite graphs, from which it follows that for every sufficiently large $n$, $\Cr(K_n) > 0.905\, H(n)$. Also using this framework, we prove that asymptotically $\Cr(K_n)$ is at least $0.985\, H(n)$. We also show that the spherical geodesic crossing number of $K_n$ is asymptotically at least $0.996\, H(n)$.
\end{abstract}

\section{Introduction}\label{sec:intro}

A long standing open problem in topological graph theory is to determine the crossing number of the complete graph $K_n$. We recall that the {\em crossing number} $\Cr(G)$ of a graph $G$ is the minimum number of pairwise crossings of edges in a drawing of $G$ in the plane. 

\subsection{Our main results}\label{subsec:mainresult}

As narrated in the illustrative survey by Beineke and Wilson~\cite{Beineke}, the problem of estimating the crossing number of complete graphs seems to have been first explored by the British artist Anthony Hill in the late 1950s. Hill found a construction that yields a drawing of $K_n$ with exactly $\frac{1}{4}\floor{\frac{n}{2}}\floor{\frac{n-1}{2}}\floor{\frac{n-2}{2}}\floor{\frac{n-3}{2}}$ crossings, for every integer $n \ge 3$~\cite{HararyHill}. In that paper, the following conjecture was put forward:

\bigskip
\noindent{\bf Conjecture.} (Hill's conjecture) \\
\[
\Cr(K_n) = H(n) := \frac{1}{4}\Bigfloor{\frac{n}{2}}\Bigfloor{\frac{n-1}{2}}\Bigfloor{\frac{n-2}{2}}\Bigfloor{\frac{n-3}{2}}.
\]
\bigskip

As we recall below in our discussion of previous work, Hill's conjecture has been only verified for $n\le 12$, and it follows from work by Norin and Zwols~\cite{Norin} that $\lim_{n\to\infty}\Cr(K_n)/H(n) > 0.905$. Our main result in this paper is the following.

\begin{theorem}\label{thm:maintheorem}
\[
\lim_{n\to\infty} \frac{\Cr(K_n)}{H(n)} > 0.98559895.
\]
\end{theorem}

We also investigate spherical drawings of $K_n$. We recall that in a {\em spherical geodesic} drawing of a graph, the host surface is the sphere, and each edge is a minimum distance geodesic arc joining its endpoints. The {\em spherical geodesic crossing number} $\Cr_{S^2}(G)$ of a graph $G$ is the minimum number of crossings in a spherical geodesic drawing of $G$. This crossing number variant is of interest not only naturally in its own, but also by its connection, unveiled by Wagner~\cite{Wagner}, to the Spherical Generalized Upper Bound Conjecture.

We note that Hill's conjecture also applies to spherical geodesic drawings, since Hill's construction can be realized as a spherical geodesic drawing. Using analogous techniques as in the proof of Theorem~\ref{thm:maintheorem}, we show the following.

\begin{theorem}\label{thm:spherical}
\[
\lim_{n\to\infty} \frac{\Cr_{S^2}(K_n)}{H(n)} > 0.99635588.
\]
\end{theorem}

Actually we prove this last bound not only for spherical geodesic drawings, but for the more general class of {\em convex} drawings~\cite{convex,convex2}. A drawing $D$ of $K_n$ in the sphere is {\em convex} if, for every 3-cycle $C$, there is a closed disc $\Delta$ bounded by $C$ with the following property: for any two vertices $u, v$ contained in $\Delta$, the edge $uv$ is contained in $\Delta$. (See Figure~\ref{fig:fighk} for examples of nonconvex drawings). We prove that the bound in Theorem~\ref{thm:spherical} holds for convex drawings. Thus in particular it holds for spherical geodesic drawings, as it is easy to see that these drawings are convex.

\begin{figure}[ht!]
\begin{center}
\includegraphics[width=8cm]{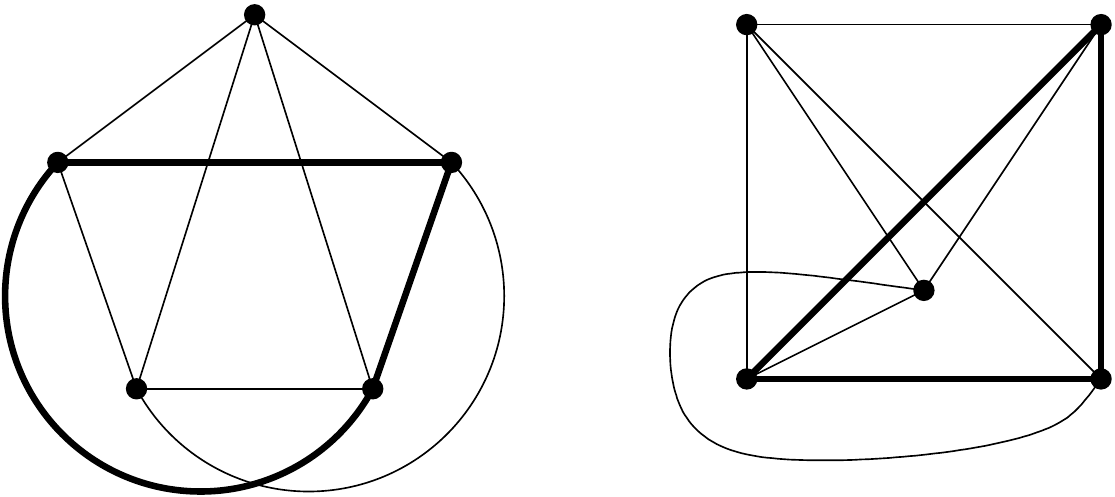}
\caption{We illustrate the two (up to isomorphism) good drawings of $K_5$ that are nonconvex. We remark that, even though not explicitly illustrated here, the host surface of these drawings is the sphere. In each case we illustrate with thick edges a $3$-cycle that witnesses that the drawing is nonconvex.}
\label{fig:fighk}
\end{center}
\end{figure}

\subsection{Previous work on Hill's conjecture}\label{subsec:previous}

We are aware of three distinct constructions that yield drawings of $K_n$ with exactly $H(n)$ crossings. Hill's construction~\cite{HararyHill} produces {\em cylindrical} drawings, which are drawings in which the vertices are drawn on two concentric circles, and no edge intersects any of these circles, except at its endpoints. Bla\v zek and Koman's construction~\cite{BlazekKoman} yields $2$-{\em page drawings} of $K_n$, that is, drawings in which every vertex lies on the $x$-axis, and each edge lies (except for its endpoints) either in the upper or in the lower halfplane. 
Very recently, 
%\'Abrego et al.~\cite{NonShellable} 
\'{A}brego, Aichholzer, Fern\'{a}ndez-Merchant, Ramos, and Vogtenhuber~\cite{NonShellable}
described a variant of Hill's construction that yields drawings of $K_n$ with $H(n)$ crossings, for every odd $n\ge 11$. 
We refer the reader to Figures 13.3 and 13.4 in~\cite{SzekelySurvey} for lively descriptions of the cylindrical drawings devised by Hill and the $2$-page drawings of Bla\v zek and Koman. For readers convenience, we reproduce these figures in Figures~\ref{fig:13-3} and \ref{fig:13-4}.

\begin{figure}[ht!]
\begin{center}
\vc{\includegraphics[width=4cm,page=1]{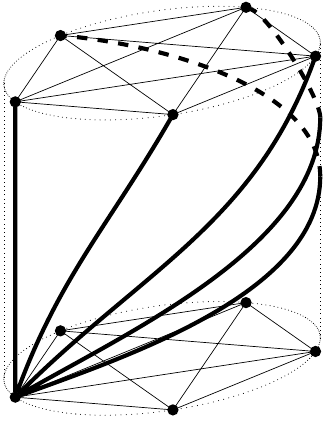}}
\hskip 2cm
\vc{\includegraphics[width=4.5cm,page=2]{fig-13-3.pdf}}
\caption{A cylindrical drawing of a $K_{10}$ is depicted on the left. The top of the cylinder is rotated for better visualization and only edges from one vertex from the bottom are drawn. The corresponding cylindrical drawing of $K_{10}$ in the plane is on the right.}
\label{fig:13-3}
\end{center}
\end{figure}

\begin{figure}[ht!]
\begin{center}
\vc{\includegraphics[width=14cm,page=1]{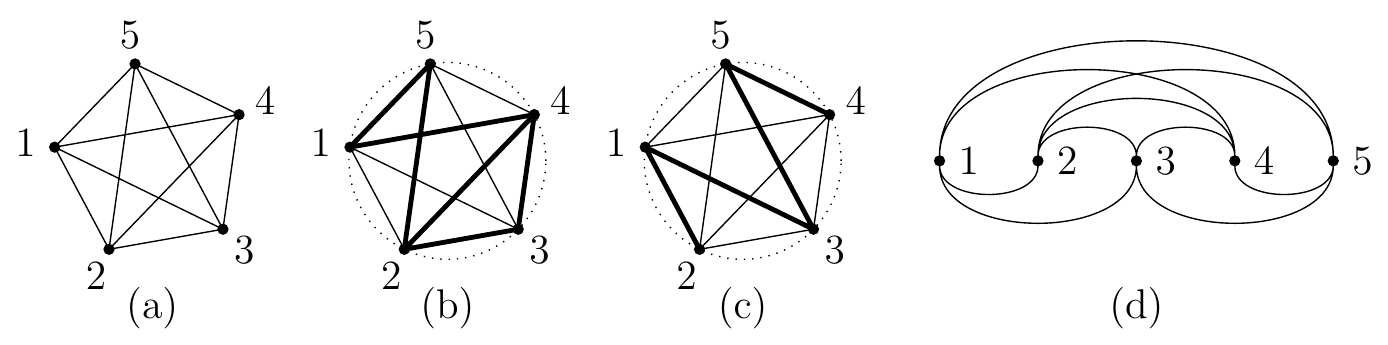}}
\caption{ The Bla\v{z}ek and Koman drawing of $K_5$. (a) A drawing of $K_5$ in the plane. (b) Edges with positive slope are drawn in the upper halfplane. (c) Edges with negative slope are drawn in the lower halfplane. (d) The actual drawing in the usual $2$-page representation, with the edges in (b) in the upper halfplane and the edges in (c) in the lower halfplane.}
\label{fig:13-4}
\end{center}
\end{figure}

Hill's conjecture has been verified both for $2$-page~\cite{2Page} and for cylindrical~\cite{Cylindrical} drawings. It is also known that the conjecture holds for {\em monotone} drawings, that is, drawings in which each edge is drawn as an $x$-monotone curve~\cite{Monotone2,BalkoMonotone}. The new construction in~\cite{NonShellable} yields drawings that are neither $2$-page nor cylindrical, but they satisfy a property called {bishellability}. 
%\'Abrego et al.~\cite{Bishellable} \todo{names} have proved that Hill's conjecture holds for bishellable drawings. 
In \cite{Bishellable}, it was proved that Hill's conjecture holds for bishellable drawings. 
 This last result implies Hill's conjecture for $2$-page, cylindrical, and monotone drawings, as all these classes of drawings are bishellable.

A straightforward counting argument shows that if Hill's conjecture holds for some odd $n$, then it also holds for $n+1$. In its full generality (that is, not for specific classes of drawings), the conjecture has only been verified for $n\le 12$. For $n\le 10$ this appears to have been reported first in~\cite{Guy10}; recently, McQuillan and Richter~\cite{DanBruce} gave a computer-free verification of Hill's conjecture for $n=9$ (and, by the previous observation, for $n=10$). Pan and Richter~\cite{PanRichter} gave a computer-assisted proof for $n=11$ (and hence for $n=12$). Hill's conjecture for $n\le 12$ has also been verified in~\cite{EuroCG15}. This last computer-assisted verification was done under the setting of rotation systems, a framework on which we also heavily rely in this work.

The conjecture for $n=13$ states that $\Cr(K_{13})=225$. An elementary counting using $\Cr(K_{11})=H(11)=100$ shows that $\Cr(K_{13}) \ge 217$. 
McQuillan, Pan, and Richter~\cite{K13} have ruled out the possibility that $\Cr(K_{13})=217$, and since $\Cr(K_{13})$ is an odd number~\cite{Parity}, it follows that $\Cr(K_{13})\in \{219,221,223,225\}$. This was further narrowed in \cite{EuroCG15}, finding that $\Cr(K_{13}) \in \{223,225\}$.

An elementary counting using that $\Cr(K_{13}) \ge 223$ shows that $\Cr(K_n) \ge \frac{223}{17160}n(n-1)(n-2)(n-3) > (0.8317+o(1))\ H(n)$. However, the best general lower bounds known for $\Cr(K_n)$ are obtained by exploiting the close relationship between the crossing numbers of complete and complete bipartite graphs.

Recall that Zarankiewicz's conjecture states that $\Cr(K_{p,q}):=Z(p,q):= \bigfloor{\frac{p}{2}}\bigfloor{\frac{p-1}{2}}\bigfloor{\frac{q}{2}}\bigfloor{\frac{q-1}{2}}$, for all positive integers $p,q$~\cite{Beineke,Decline,Turan}. It follows from a result in~\cite{RichterThomassen} that
\begin{equation}\label{eq:eq0}
{L_1} := \lim_{n\to\infty} \frac{\Cr(K_{n,n})}{Z(n,n)} \text{\hglue 0.4 cm {\rm and} \hglue 0.4 cm}  {L_2} := \lim_{n\to\infty} \frac{\Cr(K_n)}  {H(n)}
\end{equation}
both exist, and that $L_2 \ge L_1$.

A counting argument using that $\Cr(K_{5,n})=Z(5,n)$~\cite{Kleitman} implies that $L_1 \ge 0.8$. De~Klerk, Maharry, Pasechnik, Richter, and Salazar~\cite{DeKlerk1} used semidefinite programming (SDP) techniques to give a lower bound on $\Cr(K_{7,n})$, from which it follows that $L_1 > 0.83$. De~Klerk, Pasechnik, and Schrijver~\cite{DeKlerk2} also used SDP to give a lower bound on $\Cr(K_{9,n})$, and from this bound it follows that $L_1 > 0.859$. We also note that for each fixed integer $m\ge 3$, it is a finite problem to decide whether or not Zarankiewicz's conjecture holds for $K_{m,n}$, for every $n\ge m$~\cite{christian}.

Norin and Zwols (unpublished; see~\cite{Norin}) used flag algebras to show that $L_1 > 0.905$. By (\ref{eq:eq0}) this implies that $\lim_{n\to\infty} {\Cr(K_n)}/{H(n)} > 0.905$. Prior to our work, this was the best asymptotic lower bound known for $\Cr(K_n)$. In Section~\ref{sec:concludingremarks} we further discuss the work by Norin and Zwols, and explain why we can give better asymptotic lower bounds for $\Cr(K_n)$.

For a thorough recent survey of Zarankiewicz's and Hill's conjectures, we refer the reader to~\cite{SzekelySurvey}. 

We finish this survey of previous results with a few words on the spherical geodesic crossing number. This notion was introduced by Moon~\cite{Moon}, who proved the intriguing result that if one takes a random spherical drawing of $K_n$ ($n$ points are placed randomly in the sphere, and each pair of points is joined by a shortest geodesic arc), then the expected number of crossings, divided by $H(n)$, is asymptotically $1$.  
This gives a very rich set of asymptotic extremal examples. In such problems, it seems to be difficult to obtain an exact asymptotic bound using flag algebra methods.
As far as we know, the best lower bound previously known for $\Cr_{S^2}(K_n)$ is the same (asymptotically at least $0.905$) as for $\Cr(K_n)$. 

\subsection{An overview of our strategy}

Our proof makes essential use of flag algebras. This powerful tool, introduced by Razborov~\cite{Raz07}, has been the basis of several recent groundbreaking results in a variety of combinatorial and geometric problems, such as~\cite{BaberTalbot2011,BaloghC5,RBT,tripartite,HatamiHKNR13,KralMachSereni2012,PikhurkoV:2013,Raz08}, to name just a few.

Although developed in a more general setting, flag algebras in particular provide a formalism to tackle combinatorial problems of an extremal nature, in which a result of an asymptotic nature is seeked. Using flag algebras, one can find asymptotic estimates on the density of combinatorial objects, given some information on the structure of these objects for small size instances.

In a nutshell, to prove Theorem~\ref{thm:maintheorem} we exploit the fact that we have a complete understanding of all the good drawings of $K_7$~\cite{EuroCG15}, and thus of their rotation systems. (In Section~\ref{subsec:den} we review the notions of a good drawing and of a rotation system). With this information, using flag algebras we show that out of the ${n\choose 4}$ drawings of $K_4$ induced from a good drawing $D$ of $K_n$ (for every $n$ sufficiently large), less than (roughly) $0.6305{n\choose 4}$ can have $0$ crossings. Therefore $D$ must have more than $(1-0.6305){n\choose 4}=0.3695{n\choose 4}$ crossings, and thus $\Cr(K_n) > 0.3695{n\choose 4}$. Theorem~\ref{thm:maintheorem} is just an equivalent way of writing this last inequality, using a more precise rounding of the actually computed bounds.

For Theorem~\ref{thm:spherical} we proceed in an analogous manner. For this case, we use that we have the full catalogue of rotation systems that are induced from convex drawings of $K_8$. We obtain that out of the ${n\choose 4}$ drawings of $K_4$ induced from a convex drawing of $K_n$, less than (roughly) $0.6272{n\choose 4}$ can have $0$ crossings.

A more detailed outline of our arguments is given in Section~\ref{sec:overview}, where besides reviewing the concepts of good drawings and rotation systems, we introduce the notion of density, which plays a fundamental role in the theory of flag algebras. In Section~\ref{sec:reduction} we state Theorems~\ref{thm:conv1} and~\ref{thm:conv2}, two results in the language of flag algebras, and show that Theorems~\ref{thm:maintheorem} and~\ref{thm:spherical}, respectively, follow as easy consequences. The rest of the paper is then devoted to the proof of Theorems~\ref{thm:conv1} and~\ref{thm:conv2}.

\section{Densities and rotation systems}\label{sec:overview}

In this section we introduce the concepts of rotation systems and densities, which are central to the proofs of Theorems~\ref{thm:maintheorem} and~\ref{thm:spherical}. We will motivate the introduction of these notions by explaining their roles in the proof. 

\subsection{Densities in drawings of $K_n$}\label{subsec:den}

We start by recalling that a drawing of a graph is {\em good} if (i) no two adjacent edges intersect, other than at their common endvertex; (ii) no two edges intersect each other more than once; and (iii) every intersection of two nonadjacent edges is a crossing, rather than tangential. %; and (iv) no three edges cross at a common point. 

It is easy to show that every crossing-minimal drawing of a graph is necessarily good. Since we will only deal with crossing-minimal drawings (and with their induced subdrawings), we will assume throughout this work that all drawings under consideration are good.

In our context, we aim to find an asymptotic lower bound for $\Cr(K_n)$. It is easy to show that if $D$ is a good drawing of $K_n$, then each of the ${n\choose 4}$ drawings of $K_4$ induced by $D$ has exactly $0$ or $1$ crossings. Each crossing appears in exactly one such $K_4$, so our aim can be stated equivalently as follows: find an asymptotic {\em upper} bound for the proportion of non-crossing $K_4$s in a drawing of $K_n$. 

Formally, for a drawing $D$ of $K_n$ let $d(\kf;D)$ denote the probability that if we choose $4$ vertices at random from $D$, the corresponding drawing of $K_4$ induced from $D$ by these $4$ vertices has $0$ crossings. Letting $\Cr(D)$ denote the number of crossings in $D$, the above definition then implies that $\Cr(D) = \bigl(1-d(\kf;D)\bigr) {n\choose 4}$. The notation $\kf$ hints to the unique (up to isomorphism) drawing of $K_4$ with $0$ crossings (see left hand side of Figure~\ref{fig:fig1}). 

Thus $0\le d(\kf;D) \le 1$ for any drawing $D$ of $K_n$ with $n\ge 4$. Since $K_5$ cannot be drawn without crossings, it follows that $d(\kf;D)<1$ if $D$ is a drawing of $K_n$ with $n=5$ (and, actually, for any integer $n\ge 5$).

An asymptotic reading of Hill's conjecture is that $\Cr(K_n)=(3/8){n\choose 4}+O(n^3)$, and so this conjecture predicts that $d(\kf;D)$ is asymptotically at most $(1-3/8)=0.625$. What we establish in this paper is that $d(\kf;D)$ is asymptotically less than (roughly) $0.6305$. Consequently, $\Cr(K_n)/{n\choose 4}$ is asymptotically greater than $1-0.6305=0.3695$. An equivalent way to say this, as stated in Theorem~\ref{thm:maintheorem}, is that $\Cr(K_n)/H(n)$ is greater than $0.3695/0.375 > 0.985$.

Our approach consists of estimating $d(\kf;D)$, where $D$ is a crossing-minimal drawing of $K_n$ for some large integer $n$, by exploiting our complete knowledge of {all} good drawings of $K_n$ for small values of $n$, and in particular for $n=7$ and $n=8$.

With Theorem~\ref{thm:maintheorem} in mind, suppose for a moment that we limit ourselves to using the information that $\Cr(K_7)=9$. From this we obtain that for every drawing $D_7$ of $K_7$ we have $d(\kf;D_7) \le \alpha:=(1-9/{7\choose 4})\approx 0.742$. This readily implies that $d(\kf;D) \le \alpha$ for every drawing $D$ of $K_n$ with $n\ge 7$. If there existed arbitrarily large such drawings $D$ with $d(\kf;D)=\alpha$, this would mean that each induced subdrawing of $K_7$ is crossing-minimal. 

This is already impossible for $n=8$: there are no drawings of $K_8$ in which each induced subdrawing of $K_7$ has exactly $9$ crossings. Loosely speaking, it is not possible to ``pack'' $8$ crossing-minimal drawings of $K_7$ into a drawing of $K_8$. Our approach to get the much better estimate $d(\kf;D)< 0.6305$ (for large $n$) is to take the full catalogue of {\em all} good drawings of $K_7$, and use flag algebras to investigate how these can be packed into a good drawing of $K_n$, for large $n$.

\subsection{Rotation systems}\label{subsec:rot}

To achieve this last goal, we start by turning the topological problem at hand into a combinatorial one. Instead of considering directly drawings of complete graphs, we work with {\em rotation systems}. A rotation system combinatorially encodes valuable information of a drawing, by recording, for each vertex $v$, the cyclic order in which the edges incident with $v$ leave $v$ (see Figure~\ref{fig:fig1}). Thus the rotation system of a drawing of $K_n$ is a collection of $n$ cyclic permutations. In general, an {\em abstract rotation system}~\cite{kyncl1} on a set $S$ of $n$ elements is a collection of $n$ cyclic permutations, where each element $s\in S$ has an assigned cyclic permutation of the other $n-1$ elements, the {\em rotation} at $s$. We often use $s{:}s_1 s_2 \ldots s_{n-1}$ to denote that the cyclic permutation assigned to $s$ is $s_1s_2\ldots s_{n-1}$. We say that $S$  is the {\em ground set} of the abstract rotation system.

Throughout this work, for brevity, we shall refer to an abstract rotation system simply as a rotation system.

\begin{center}
\begin{figure}[ht]
\scalebox{0.95}{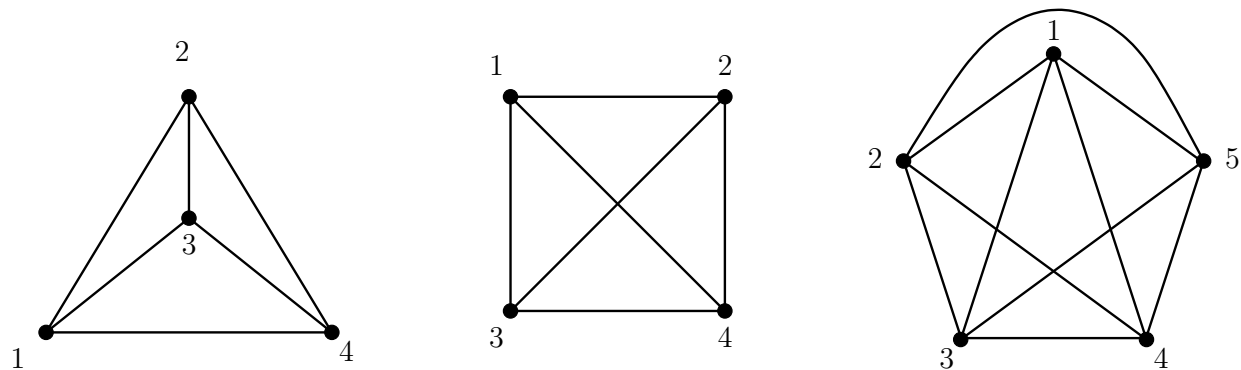}
\caption{The left hand side drawing of $K_4$ induces the rotation system $N_4:=\{1{:}234,2{:}143,3{:}124,4{:}132\}$. The drawing of $K_4$ in the center induces the rotation system $\{1{:}243, 2{:}143, 3{:}124,4{:}123\}$. The drawing $D_3$ of $K_5$ on the right hand side induces the rotation system $\{  1{:}2543, 2{:} 1435, 3{:}1542, 4{:}1532, 5{:}1243\}$. We remark that since the rotation at each vertex is a {\em cyclic} permutation of the other vertices, we may alternatively write this last rotation system, for instance, as $\{ 1{:}3254, 2{:} 3514, 3{:}1542, 4{:}2153, 5{:}3124\}$. 
}
\label{fig:fig1}
\end{figure}
\end{center}

Two rotation systems are {\em isomorphic} if each of them can be obtained from the other simply by a relabelling of its elements. An abstract rotation system is {\em realizable} (respectively, {\em convex}) if it is isomorphic to the rotation system induced by a good drawing of $K_n$ (respectively, by a convex drawing of $K_n$). Every convex rotation system is realizable, as the set of convex drawings is a (proper) subset of the collection of good drawings.

Given a rotation system $R$ on a set $S$ of $n$ elements, and a subset $S'$ of $S$, $R$ naturally induces a rotation system (a {\em rotation subsystem}) on $S'$, simply by removing from $R$ all the appearences of the elements in $S\setminus S'$. For instance, if $R$ is the rotation system $\{ 1{:}234, 2{:}143, 3{:}142, 4{:}123\}$ on $S=\{1,2,3,4\}$, and we let $S'=\{1,2,4\}$, then the rotation system on $S'$ induced by $R$ is $R'=\{ 1{:}24, 2{:}14, 4{:}12\}$.

\subsection{Densities in rotation systems}\label{subsec:denrot}

The notion of density of $\kf$ in a drawing of $K_n$ gets naturally extended to rotations. In general, if $R,R'$ are rotation systems, then we let $d(R';R)$ denote the probability that a randomly chosen rotation system of $R$ with $|R'|$ elements is isomorphic to $R'$. Note that if $|R'|>|R|$, then $d(R';R)=0$.

There is (up to isomorphism) a unique rotation system $N_4$ on $4$ elements induced by a drawing of $K_4$ with no crossings; again we refer the reader to Figure~\ref{fig:fig1}, in whose caption $N_4$ is presented.

For a (realizable or not) rotation system $R$ on $n\ge 4$ elements, let $d(N_4;R)$ denote the probability that a randomly chosen rotation subsystem of $R$ with $4$ elements is isomorphic to $N_4$. Clearly, if $R$ is realized by a drawing $D$ of $K_n$, then $d(\kf;D)=d(N_4;R)$. Thus, in order to prove Theorem~\ref{thm:maintheorem}, it suffices to show that $d(N_4;R)<0.6305$ for every sufficiently large realizable rotation system $R$. For Theorem~\ref{thm:spherical}, we show that the bound $d(N_4;R)<0.6272$ holds if $R$ is convex.

We know the family $\ee_7$ of $22{,}730$ non-isomorphic realizable rotation systems on $7$ elements (this is discussed in Section~\ref{sec:real7}). A trivial, but key observation, is that if $R$ is a realizable rotation system on $n\ge 7$ elements, then each of the rotation subsystems of $R$ on $7$ elements is (isomorphic to a rotation) in $\ee_7$.

What we show is that {\em if $R$ is a realizable rotation system on $n$ elements {such that each of its rotation subsystems on $7$ elements is in $\ee_7$}, then $d(N_4;R)<0.6305$} (as long as $R$ is sufficiently large). We show this by using tools from the flag algebra framework. The size $22{,}730$ turns out to be small enough to be handled with these techniques.

For Theorem~\ref{thm:spherical} we proceed in a similar way. The improvement over the general bound in Theorem~\ref{thm:maintheorem} is obtained using the set $\cc_8$ of convex realizable systems, which is also small enough ($7{,}360$ rotations) to use the flag algebra approach.

\section{Convergent subsequences of rotation systems: \\ proof of Theorems~\ref{thm:maintheorem} and~\ref{thm:spherical}}\label{sec:reduction}

In this section we show that Theorems~\ref{thm:maintheorem} and~\ref{thm:spherical} follow from two results on sequences of rotation systems. These statements involve the notion of convergence, from the flag algebra framework. 

Let $R_1,R_2,\ldots,$ be an infinite sequence of rotation systems, where $|R_i|<|R_{i+1}|$ for $i=1,2,\ldots$. The sequence $R_1,R_2,\ldots$ is {\em convergent} if, for each fixed rotation system $R'$, the sequence $\{d(R';R_i)\}_{i=1}^\infty$ converges.

A standard compactness argument, using Tychonoff's theorem, shows that every infinite sequence of rotation systems has a convergent subsequence. In particular, there exist convergent sequences of realizable, and of convex, rotation systems. Such convergent sequences are the central objects in the next statements which, as we shall see shortly, easily imply Theorems~\ref{thm:maintheorem} and~\ref{thm:spherical}, respectively.%\todo{Are these the right numbers? Taken from results.txt. Version sent by Bernard had different numbers. BL: These numbers are correct, just the last digit was not right - one needs to round up.}

\medskip

\begin{theorem}\label{thm:conv1}
Let $R_1, R_2, \ldots$ be a convergent sequence of realizable rotation systems. Then 
\[
\lim_{i\to\infty} d(N_4;R_i) < A :=  \frac{22064013752809590266065131421016}{35000000000000000000000000000000} < 0.630400393.
\]
\end{theorem}

\medskip

\begin{theorem}\label{thm:conv2}
Let $R_1, R_2, \ldots$ be a convergent sequence of convex rotation systems. Then 
\[
\lim_{i\to\infty} d(N_4;R_i) < B := \frac{43909978466574504806937629255000}{70000000000000000000000000000000} < 0.627285407.
\]
\end{theorem}

\medskip

The rest of this paper will be devoted to the proofs of these statements. We close this section by showing how Theorem~\ref{thm:maintheorem} follows from Theorem~\ref{thm:conv1}. The proof that Theorem~\ref{thm:spherical} follows from Theorem~\ref{thm:conv2} is analogous.

\begin{proof}[Proof of Theorem~\ref{thm:maintheorem}, assuming Theorem~\ref{thm:conv1}]

Let $D_{1}, D_{2},\ldots$ be an infinite sequence of drawings such that, for $i\in\mathbb{N}$, $D_i$ is a crossing-minimal drawing of $K_i$. For $i\in\mathbb{N}$, let $R_i$ be the rotation system induced by $D_i$.

A well-known argument using Tychonoff's theorem shows that $R_1,R_2,\ldots$ has a convergent subsequence $R_{n(1)}, R_{n(2)},\ldots$. Since (as observed in Section~\ref{subsec:denrot}) $d(N_4;R_{n(i)}) = d(\kf;D_{n(i)})$ for $i=1,2,\ldots$, from Theorem~\ref{thm:conv1} we have $\lim_{i\to\infty} d(\kf;D_{n(i)}) < A$. 

The crossing-minimality of each $D_{n(i)}$ means that $\Cr(K_{n(i)})=\Cr(D_{n(i)})$ for $i\in\mathbb{N}$. Now since $\Cr(D_{n(i)}) = \bigl(1-d(\kf;D_{n(i)})\bigr) {n\choose 4}$ for each $i\in\mathbb{N}$, the convergence of $d(\kf;D_{n(1)})$, $d(\kf;D_{n(2)}),\ldots$ to a number smaller than $A$ implies that
\begin{equation}\label{eq:eq1}
\frac{\Cr(K_{n(1)})}{{n(1)\choose 4}}, \frac{\Cr(K_{n(2)})}{{n(2)\choose 4}},\ldots
\end{equation}
is a convergent sequence, whose limit is greater than $1- A$.

Since the sequence in (\ref{eq:eq1}) is a subsequence of the sequence $\{\Cr(K_n)/{n\choose 4}\}_{n=1}^\infty$, and this sequence is also convergent~\cite{RichterThomassen}, then $\lim_{n\to\infty}  \Cr(K_n)/{n\choose 4} > 1-A$. Since $\lim_{n\to\infty} H(n)/{n\choose 4} = 3/8= 0.375$, then $\lim_{n\to\infty} \Cr(K_n)/H(n) > (1-A)/0.375 > 0.98559895$.
\end{proof}

%%%%%%%%%%%%%%%%%%%%%%%%%%%%%%%%%%%%%

%%%%%%%%%%%%%%%%%%%%%%%%%%%%%%%%%%%%

\section{Small rotation systems}\label{sec:real7}

As described in Section~\ref{sec:overview}, an essential ingredient in the proof of Theorem~\ref{thm:conv1} is that we know the full collection of all non-isomorphic realizable rotation systems on $7$ elements. Analogously, to prove Theorem~\ref{thm:conv2} we use the collection of all non-isomorphic convex rotation systems on $8$ elements.

In this section we describe how these families are obtained. 

\subsection{Realizable rotation systems on $7$ elements}

For each integer $n \ge 3$, we use $\ee_n$ to denote the set of all non-isomorphic realizable rotation systems on $n$ elements.

Aichholzer and Pammer wrote code to obtain all non-isomorphic realizable rotation systems on $n$ elements for $n\le 9$, with the results reported in~\cite[Table 1]{EuroCG15} (see also~\cite{pammer}). We note that in~\cite{EuroCG15} a different notion of isomorphism (to the one used in this paper) is used. Let us say that two rotation systems $R,R'$ are {\em equivalent} if either $R$ is isomorphic to $R'$, or if $R'$ is isomorphic to the system obtained by taking the inverse of each of the rotations in $R$ (that is, if $R'$ is the {\em inverse} $R^{-1}$ of $R$). Under this terminology, in~\cite{EuroCG15} the collections of non-equivalent realizable rotation systems on $n$ elements were reported, for all $n\le 9$.

Thus the set $\mm_n$ of non-equivalent realizable rotation systems on $n$ elements can be obtained from $\ee_n$: if for some rotation $R$, both $R$ and $R^{-1}$ are in $\ee_n$, we remove one of them. Similarly, $\ee_n$ can be easily obtained from $\mm_n$. First grow $\mm_n$ by adding the inverse of each of its elements, and then run an isomorphism check to get rid of duplicates.

\begin{center}
\begin{figure}[ht!]
\centering
\scalebox{0.5}{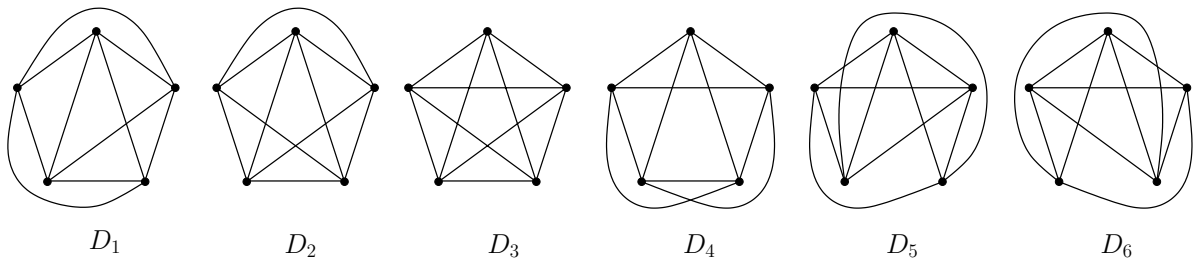}
\caption{The six non-isomorphic drawings of $K_5$. Here we adopt the point of view that two drawings are isomorphic if there is an orientation-preserving self-homeomorphism of the sphere that takes one into the other. If we dropped the orientation-preserving condition, then $D_5$ and $D_6$ would be isomorphic.}
\label{fig:fig2}
\end{figure}
\end{center}

We wrote code to obtain $\ee_7$, proceeding as follows. First we obtain $\ee_5$. To achieve this, it suffices to take the collection of non-isomorphic drawings of $K_5$. Here we use the notion that two drawings are {\em isomorphic} if there is an orientation-preserving self-homeomorphism of the plane that takes one into the other. An easy exercise shows that there are exactly six non-isomorphic drawings of $K_5$, namely the ones depicted in Figure~\ref{fig:fig2}. The class $\ee_5$ consists of the rotation systems that correspond to these drawings.

Aichholzer (private communication) noted, based on his results, that a rotation system on $6$ elements is realizable if and only if each of its rotation subsystems on $5$ elements is realizable. As Kyn\v{c}l observed in~\cite[Sect.~1]{kynclsimplified}, it follows from this observation and~\cite[Theorem 1]{kynclsimplified} that a rotation system on $n\ge 5$ elements is realizable if and only if each of its rotation subsystems on $5$ elements is realizable.

From this last important observation it follows that the task of finding $\ee_6$ is straightforward. For each rotation in $\ee_5$, we try all possible ways to extend it to a rotation system on $6$ elements, and for each of these possible ways, we test whether or not each of its rotation subsystems on $5$ elements is in $\ee_5$. Finally, we perform an isomorphism check to get rid of duplicates, and finally obtain $\ee_6$. To obtain $\ee_7$ from $\ee_6$ we follow an analogous procedure.

The family $\ee_6$ has $165$ elements, and $\ee_7$ has $22{,}730$ elements. From these lists we generated $\mm_6$ and $\mm_7$, which have $102$ and $11{,}556$ elements, respectively. These coincide with the collections reported in~\cite[Table 1]{EuroCG15}, as kindly verified by Aichholzer (private communication). The sets $\ee_6$  and $\ee_7$ are available at \oururl.

%one can easily 

%As we mentioned above, from $\mm_6$ one can easily generate $\ee_6$.

%Before our calculation of $\ee_6$ Aichholzer sent us the list of $102$ elements of $\mm_6$.  Running this algorithm on the collection sent to us by Aichholzer yields the same set $\ee_6$ we obtained. This expected confirmation verified the correctness of our code to generate $\ee_{n}$ from $\ee_{n-1}$.

%Seeking an additional verification, from $\ee_7$ we obtained $\mm_7$, finding a list of $11{,}556$ rotation systems. We sent this collection to Aichholzer, who kindly verified that this is indeed the same set previously reported in~\cite[Table 1]{EuroCG15}. 

\subsection{Convex rotation systems on $8$ elements}

Arroyo, McQuillan, Richter, and Salazar~\cite{convex} have characterized convex drawings of $K_n$ as follows. A good drawing $D$ of $K_n$, with $n\ge 5$, is convex if and only if all its induced drawings of $K_5$ are isomorphic to rectilinear drawings. It is well-known that up to isomorphism there are three such drawings of $K_5$, namely $D_1,D_2$, and $D_3$ in Figure~\ref{fig:fig2}. 

Thus in order to generate the collection $\cc_n$ of convex rotation systems, for $n\ge 5$, it suffices to follow the procedure described above to obtain $\ee_n$, but in this case the foundation $\cc_5$ consists of the rotation systems that correspond to $D_1,D_2$, and $D_3$. In this way we constructed $\cc_6,\cc_7$, and $\cc_8$. This last collection consists of $7{,}360$ rotation systems, thus being even more manageable, for a flag algebra treatment, than $\ee_7$.

We note that we do not really need the full characterization from \cite{convex}. We only need the easy ``only if'' part, which is readily verified by checking that $D_4,D_5$, and $D_6$ are not convex. If we did not have the ``if'' part, we would still know that the class $\cc_8$ we constructed contains the class of convex drawings. Thus our results, in particular Theorem~\ref{thm:spherical}, would still hold without this non-trivial direction of the characterization from \cite{convex}.

\section{Flag algebras}\label{sec:flag}

This section contains a brief introduction to the flag algebra framework, in the setting of rotation systems. 
For a more detailed and general exposition, see the original paper of Razborov~\cite{Raz07}.
For more accessible introductions to flag algebras, see for instance~\cite{BaberTalbot2011,interim}.
%Readers familiar with flag algebras may safely skip this section. NOT REALLY, I think (GS). 
%We write the description for the rotation systems but it can be generalized to many different settings.
%We do not provide a complete description of the framework. 

Throughout this discussion $\cF$ is an infinite set of rotation systems, and for each $\ell\in\Nat$, $\cF_\ell$ is the set of all rotations in $\cF$ with $\ell$ elements. For our cases of interest, in the next section we will take $\rr$ to be the collection $\ee$ of all realizable rotation systems (to prove Theorem~\ref{thm:maintheorem}), or the collection $\cc$ of all convex rotation systems (to prove Theorem~\ref{thm:spherical}). 

For $R \in\cF_\ell$ and $R' \in\cF_{\ell'}$, define  $p(R,R')$ to be the probability that choosing $\ell$ vertices
uniformly at random from $R'$ induces a rotation isomorphic to $R$. Note that $p(R,R') = 0$ if $\ell > \ell'$. 

For $R \in \cF$, we denote by $V(R)$ the ground set of $R$. We use $V(R)$ to hint that we think of the ground elements of $R$ as, and call them, {\em vertices} (after all, we are interested in rotation systems that are induced by drawings of $K_n$). Although evidently $R$ is not a graph, the rotation systems that we will investigate come from drawings of $K_n$, and as such, have an identity as vertices. We let $v(R):=|V(R)|$. Note that $v(R)$ is also the number of elements (cyclic permutations) of $R$.

We start by defining algebras $\cA$ and $\cA^{\sigma}$, where $\sigma$ is any rotation system in $\cF$.
These algebras will be called \emph{flag algebras}.
Let $\mathbb{R}\cF$ be the set of all formal linear combinations of elements in $\cF$ with real coefficients.
Furthermore, let $\cK$ be the linear subspace generated by all linear combinations of the form
\begin{align}
R - \sum_{R' \in \cF_{v(R)+1}} p(R,R') \cdot R'. \label{eq:zero}
\end{align}
We define $\cA$ as the space $\mathbb{R}\cF$ factorized by $\cK$.
The space $\cA$ comes with naturally defined operations of addition, and multiplication by a real number.
To introduce the multiplication in $\cA$, we first define multiplication of two elements in $\cF$.
For $R_1, R_2 \in \cF$, and $R\in\cF_{v(R_1)+v(R_2)}$,
we define $p(R_1, R_2; R)$ to be the probability that for a randomly chosen subset $I_1$ of $V(R)$
of size $v(R_1)$, the rotation subsystems of $R$ induced by $I_1$ and $I_2:=V(R)\setminus I_1$ are isomorphic to $R_1$ and $R_2$, respectively. We set
\[R_1 \times R_2 = \sum_{R\in\cF_{v(R_1)+v(R_2)}}p(R_1,R_2;R) \cdot R.\]
The multiplication in $\cF$ has a unique linear extension to $\Real\cF$, which yields
a well-defined multiplication also in $\cA$. 
A formal proof of this can be found in~\cite[Lemma 2.4]{Raz07}.

%\todo{JB:  you call theta "injective mapping", but later theta sounds like that the actual image of $V(\rho)$. Can you "unify" the two concepts, or change the first a bit? Or just expand explanation?}
Now we introduce an algebra $\cA^\sigma$ for each $\sigma \in \cF$.
The element $\sigma$ is usually called a \emph{type} within the flag algebra framework.
Without loss of generality, assume that the vertices of $\sigma$ are labelled $1,2,\ldots,v(\sigma)$.
Define $\cF^\sigma$ to be the set of all elements in $\cF$ with a fixed \emph{embedding} of
$\sigma$, i.e., an injective mapping $\theta$ from $V(\sigma)$ to $V(R)$ such that
the image of $\theta$, denoted by $\theta(V(\sigma))$, induces in $R$ a rotation isomorphic to $\sigma$.
Following the customary flag algebra terminology, the elements of $\cF^{\sigma}$ are \emph{$\sigma$-flags}, 
and the rotation induced by $\theta(V(\sigma))$ is the \emph{root} of a $\sigma$-flag.
%We will now define $\cA^\sigma$ analogously as $\cA$, but all flags are $\sigma$-flags.
%For more detailed explanation, see~\cite{Raz07}.
%\todo{write more here}

For every $\ell\in\Nat$, we define
$\cF^{\sigma}_\ell\subset \cF^{\sigma}$ to be the set of the $\sigma$-flags from
$\cF^{\sigma}$ that have size $\ell$.
Analogously to the case for $\cA$, for two $\sigma$-flags $R, R' \in\cF^{\sigma}$ with embeddings of $\sigma$ given by $\theta, \theta'$, we set
$p(R,R')$ to be the probability that a randomly chosen subset of $v(R)-v(\sigma)$ ground elements in
$V(R')\setminus\theta'(V(\sigma))$ together with $\theta'(V(\sigma))$ induces a
substructure that is
isomorphic to $R$ through an isomorphism $f$ that preserves the embedding of $\sigma$.
In other words, the  
isomorphism $f$ has to satisfy $f(\theta') = \theta$.
Let $\Real\cF^{\sigma}$ be the set of all formal linear combinations of elements
of $\cF^\sigma$ with real coefficients, and let $\cK^\sigma$ be the linear subspace
of $\Real\cF^\sigma$ generated by all the linear combinations of the form
\[R-\sum_{R'\in\cF^\sigma_{v(R)+1}}p(R,R')\cdot R'.\]
We define $\cA^\sigma$ to be $\Real\cF^\sigma$ factorized by $\cK^\sigma$.

We now proceed to define the multiplication of two elements from $\cF^\sigma$. 
Let $R_1, R_2\in \cF^\sigma$, $R\in \cF^\sigma_{v(R_1)+v(R_2)-v(\sigma)}$, and let $\theta$ be the fixed embedding of $\sigma$ in $R$.
%As in the definition of multiplication for $\cA$, we define $p(R_1, R_2; R)$ as some probability.
Choose uniformly at random a subset of $X$ in $V(R)\setminus \theta(V(\sigma))$ of size
$v(R_1)-v(\sigma)$. Let $Y = V(R)\setminus \{ \theta(V(\sigma))  \cup  Y\}$ of
size $v(R_2)-v(\sigma)$.
We define $p(R_1, R_2; R)$ to be the probability that  $X \cup \theta(V(\sigma))$ and $Y \cup \theta(V(\sigma))$ induce rotations isomorphic to $R_1$ and $R_2$, respectively.  
This definition naturally extends to $\cA^\sigma$.

Consider an infinite sequence $(R_n)_{n \in \mathbb{N}}$, where $R_n \in \cF_n$. We note that the density $d(R; R_n)$ used in Section~\ref{sec:reduction} is simply $p(R,R_n)$ in the current setting. We use $p(R,R_n)$ in this section as this is the custom notation in flag algebra discussions. We recall from Section~\ref{sec:reduction} that $(R_n)_{n \in \mathbb{N}}$ is \emph{convergent}  
if the sequence $\bigl(p(R,R_n)\bigr)_{n\in\mathbb{N}}$ converges for every $R \in \cF$.
A standard compactness argument using Tychonoff's theorem yields that every
infinite sequence has a convergent subsequence. 
Fix a convergent sequence  $(R_n)_{n \in \mathbb{N}}$. 
For every $R \in \cF$, we set $\phi(R) = \lim_{n \to \infty} p(R, R_n)$ and linearly extend
$\phi$ to $\cA$.
We usually refer to the mapping $\phi$ as the \emph{limit of the sequence}.
The obtained mapping $\phi$ is a homomorphism from $\cA$ to $\mathbb{R}$.
Note that for every $R\in \cF$ we have $\phi(R)\geq 0$. 
Let $\Hom^+(\cA, \mathbb{R})$ be the set of all such homomorphisms, i.e., the set of all homomorphisms
$\psi$ from the algebra $\cA$ to $\mathbb{R}$ such that $\psi(R)\ge0$ for every $R\in\cF$.
An interesting, crucial fact in the theory of flag algebras, is that this set is exactly the set of all limits of convergent sequences in $\cF$~\cite[Theorem~3.3]{Raz07}.

It is possible to define a homomorphism  $\phi^\sigma$ from $\cA^\sigma$ to $\mathbb{R}$
and an \emph{unlabelling operator} $\llbracket \cdot \rrbracket_\sigma : \cA^\sigma \to \cA$
such that if $\phi^\sigma(A^\sigma) \geq 0$ for some $A^\sigma \in \cA^\sigma$, then 
$\phi(\llbracket A^\sigma \rrbracket_\sigma) \geq 0$. 
For details, see~\cite{Raz07}. 
The unlabelling operator is very useful for generating non-obvious valid inequalities
of the form $\phi(A) \geq 0$ for some $A \in \cA$. 
In particular, $\phi(\llbracket \left( A^\sigma \right)^2  \rrbracket_\sigma) \geq 0$ is always a valid inequality, and the generation of these inequalities can be somewhat automated.

\section{Proof of Theorem~\ref{thm:conv1}\label{sec:conv1}}

\begin{proof}[Proof of Theorem~\ref{thm:conv1}]
We use the flag algebra framework developed in the previous section, performing the calculations on $\mathcal{E}_7$. As we observed in Section~\ref{sec:real7}, this set has cardinality 22,730. We follow the convention from the previous section to think of the elements in the ground set of a rotation as {\em vertices}.

We used 1803 labeled flags of 8 types $\sigma_1,\ldots,\sigma_8$.
Type $\sigma_1$ is one labeled vertex and let $F_1$ be $\mathcal{E}^{\sigma_1}_{4}$. 
Type $\sigma_2$  are three labeled vertices and let $F_2$ be $\mathcal{E}^{\sigma_2}_{5}$. 
Types $\sigma_i$ for $3 \leq i \leq 8$ are all labeled rotations on 5 vertices, namely the ones associated to the drawings in Figure~\ref{fig:fig2}.
For $3 \leq i \leq 8$, let $F_i = \mathcal{E}^{\sigma_i}_{6}$. 
Notice that for all $i$ we picked the sizes of flags in $F_i$ such that the product of any two flags from $F_i$
can be expressed in $\mathcal{E}_7^{\sigma_i}$, and hence subsequently gives an equation in $\mathcal{E}_7$.

The following holds for any $\phi \in \Hom^+(\cA, \mathbb{R})$.
Let $M_1,\ldots,M_8$ be positive semidefinite matrices, where $M_i$ has the same dimension as $F_{i}$ for all $i$.
Then 
\begin{equation}\label{eq:first}
0 \leq \phi \left(  \sum_{1 \leq i \leq 8} \llbracket F_i^T M_i F_i  \rrbracket_{\sigma_i} \right) =  \phi \left( \sum_{R \in \mathcal{E}_7} c_R \cdot R \right),
\end{equation}
where $c_R$ is a real number depending on $M_1,\ldots,M_8$ for each $R$. 
The expression \eqref{eq:zero} implies that
\[
\phi \left( \yu \right) =  \phi \left( \sum_{R \in \mathcal{E}_7} p(\yu, R) \cdot R  \right).
\]
By combining this and \eqref{eq:first} we obtain the following, where (we recall from Section~\ref{sec:overview}) $N_4$ is the rotation system that corresponds to $\kf$:
\[ %]\begin{equation}\label{eq:second}
\phi \left( \yu \right) =  \phi \left( \sum_{R \in \mathcal{E}_7} p(\yu, R) \cdot R  \right) \leq \phi \left( \sum_{R \in \mathcal{E}_7} (p(\yu, R) + c_R) \cdot R \right).
\] %\end{equation}
Let $A$ be as in the statement of Theorem~\ref{thm:conv1}. By solving an instance of a semidefinite program, we found $M_1,\ldots,M_8$ such that
\[ %]\begin{equation}\label{eq:third}
p(\yu, R) + c_R \leq A
\] %\end{equation}
for all $R \in  \mathcal{E}_7$. Noting that  $\phi \left( \sum_{R \in \mathcal{E}_7} R \right) = 1$, we obtain
\[ %]\begin{equation}\label{eq:fourth}
\phi \left( \yu \right) \leq \phi \left( \sum_{R \in \mathcal{E}_7} (p(\yu, R) + c_R) \cdot R \right)\leq A \cdot \phi \left( \sum_{R \in \mathcal{E}_7} R \right) = A.
\] %\end{equation}

Let $R_1,R_2,\ldots$ be a convergent sequence of realizable rotation systems. Since $\phi(N_4) = \lim_{i\to\infty} p(N_4,R_i) = \lim_{i\to\infty} d(N_4;R_i)$, this last equation implies that $\lim_{i\to\infty} d(N_4;R_i) \le A < 0.630400393$, as claimed in Theorem~\ref{thm:conv1}. 
Although the $M_i$s in general may have real numbers as entries, the $M_i$s in our calculation they all have rational numbers as entries.
In addition, for every $i$ we construct $M_i$ as $U^TDU$, where $D$ is a diagonal matrix with non-negative entries. 
This implies that all $M_i$s are indeed positive semidefinite and the calculations of $A$ are performed exactly over rational numbers. 

Due to space limitations, we provide $\mathcal{E}_7$, $F_i$ and $M_i$ for all $i$, as well as programs that perform the calculations, in electronic files at \oururl.
The entire calculation, including generating $M_i$s, takes about 8 hours on a high performance machine. The number of variables in the $M_i$s is 242099 and the data is about 258MB.
\end{proof}

\begin{proof}[Proof of Theorem~\ref{thm:conv2}] 
%\todo{JB: a bit dense, recalling what $C_i^sigma$ would be helpful}
In this case we performed the calculations on $\cc_8$. We used 3664 labeled flags of $5$ types  $\sigma_1,\ldots,\sigma_5$.
Type $\sigma_1$ is one labeled vertex and let $F_1$ be $\cc^{\sigma_1}_{4}$, i.e., all realizable convex rotation systems on $4$ vertices, where one vertex is labeled. 
Type $\sigma_2$  are three labeled vertices and let $F_2$ be $\cc^{\sigma_2}_{5}$. 
Types $\sigma_i$ for $3 \leq i \leq 5$ are all labeled rotations on 5 vertices, namely the ones associated to the drawings $D_1,D_2$, and $D_3$ in Figure~\ref{fig:fig2}.
For $3 \leq i \leq 5$, let $F_i = \cc^{\sigma_i}_{6}$. 
Notice that for all $i$ we picked the sizes of flags in $F_i$ such that the product of any two flags from $F_i$ 
%\todo{GS: I wrote this thinking this was the natural extension from the objects in the previous proof, are these objects ok? Looks good to BL.}
can be expressed in $\cc_8^{\sigma_i}$, and hence subsequently gives an equation in $\cc_8$.

We can now pick up the proof of Theorem~\ref{thm:maintheorem} at the paragraph that starts ``The following holds\ldots'', with the following changes. Instead of having positive semidefinite matrices $M_1,\ldots,M_8$, we have only five positive semidefinite matrices $M_1,\ldots,M_5$ (here again each $M_i$ has the same dimension as $F_i$). The first summation in \eqref{eq:first} is now on $1\le i\le 5$, and every summation on $R\in \ee_7$ gets replaced by a summation on $R\in\cc_8$. Finally, instead of the constant $A$ in Theorem~\ref{thm:maintheorem}, we have the constant $B$ in Theorem~\ref{thm:spherical}. 

With these changes the proof carries over exactly as in the previous proof, finally obtaining that $\lim_{i\to\infty} d(N_4;R_i) \le B < 0.627285406$.

Due to space limitations, we provide $\mathcal{C}_8$, $F_i$ and $M_i$ for all $i$, as well as programs that perform the calculations, in electronic files at \oururl.
The entire calculation, including generating the $M_i$s, takes 7 hours on a high performance machine.
 The number of variables in $M_i$s is 865872 and the data is about 354MB.
\end{proof}

%We observed that replacing $\mathcal{R}_7$ with $\mathcal{W}_7$ gives a slightly worse bound
%\[
%\frac{22064304}{35000000000000000000000000000000} < 0.6304086948.
%\]

\section{Concluding remarks}\label{sec:concludingremarks}

As we mentioned in Section~\ref{sec:intro}, the flag algebra framework was used by Norin and Zwols~\cite{Norin} to attack another crossing number problem, namely Zarankiewicz's conjecture. Recently, Goaoc, Hubard, De Joannis De Verclos, Sereni, and Volec~\cite{limitsorder} also used flag algebras to approach a related problem in discrete geometry, namely the density of $k$-tuples in convex position in point sets in the plane.

{Norin and Zwols computed all the good drawings of $K_{3,4}$, and for each such drawing they recorded which pairs of edges cross each other. With this information, they used flag algebras to obtain the lower bound $\lim_{n\to\infty} \Cr(K_{n,n})/Z(n,n) > 0.905$. In this paper we worked with rotation systems, but we note that this approach is equivalent to the alternative (\`a la Norin-Zwols) of computing all good drawings of $K_7$ and recording, for each such drawing, which pairs of edges cross each other. This follows since from the rotation system of a drawing one can tell which pairs of edges cross each other in the drawing~\cite{gioantheorem,gioan}.} As we mentioned in Section~\ref{subsec:previous}, the Norin-Zwols result implies the bound $\lim_{n\to\infty} \Cr(K_n)/H(n) > 0.905$. The improved bound we report in this paper is explained since the information of all good drawings of $K_7$ is remarkably more extensive than the information obtained by considering all good drawings of $K_{3,4}$. 

{An earlier approach we tried involved associating to a good drawing $\dd$ of $K_m$ the $4$-uniform hypergraph $\hh_\dd$ whose vertices are the vertices of the drawing, and where four vertices form an edge if and only if the drawing of $K_4$ induced from $\dd$ on these four vertices has a crossing. We refer the reader to~\cite[Section 13.4]{SzekelySurvey} for a discussion on the connection between crossing number problems and Tur\'an-type hypergraph problems. This approach, also using flag algebras, yielded a considerably weaker lower bound than the one in Theorem~\ref{thm:maintheorem}. Obtaining poorer bounds in this setting is quite natural since, as we recalled above, with the rotation system of a drawing one can tell not only which $K_4$s have a crossing, but exactly which edges cross each other in a given $K_4$.}

We are currently working on two separate approaches to apply flag algebras to obtain improved lower bounds on the rectilinear crossing number $\rcr{(K_{n,n})}$. We can currently show that $\lim_{n\to\infty} \rcr(K_{n,n})/Z(n,n) > 0.973$, and we hope to get an even better lower bound when a set of ongoing calculations is completed. Together with Pfender and Norin, we had previously considered the special version of rectilinear drawings in which the partite classes are separated by a line. In this case, we got a lower bound of $0.99$.

Let us mention that it might be possible to improve the constants $A$ and $B$ in Theorems~\ref{thm:conv1} and~\ref{thm:conv2} by a tiny amount. The matrices $M_i$ in the proofs of these theorems were first obtained by a semidefinite programming solver. 
These matrices do not contain exact entries, and some small rounding was necessary to ensure that the $M_i$s are indeed positive semidefinite and the evaluation of $p(\yu, R) + c_R$ does not have any numerical errors. 
We have not tried to optimize the rounding process as we think the possible improvement is negligible.  

For Theorem~\ref{thm:conv1}, performing the calculations on $\mathcal{E}_8$ would likely provide a remarkable improvement. Unfortunately, the size of this set makes it out of reach for current computers. Similarly, for Theorem~\ref{thm:conv2}, performing the calculations on $\mathcal{C}_9$ would very likely result in a considerable improvement, but this set is also too big to be handled with computer power available at this time.

Aichholzer (private communication) has verified that all crossing-minimal drawings of $K_n$, for $n\le 12$, are convex. Thus it seems reasonable to conjecture that all crossing-minimal drawings of $K_n$, for every integer $n$, are convex. If this were proved, the bound in Theorem~\ref{thm:spherical} would apply for the crossing number of $K_n$.

%\redit{Mention hereditary convex. That little improvement is obtained.}

% We are currently working on taking this one step further, as the class of all convex realizable rotation systems on $9$ elements also appears to be of a manageable size to apply flag algebras tools. We expect to further refine the bound given by Theorem~\ref{thm:spherical}, but we estimate it will take several months to complete the ongoing associated computations.
\bigskip

\section*{Acknowledgements}

We thank Oswin Aichholzer for making available to us the collection $\mm_6$, and for checking that our collection $\mm_7$ agrees with the one previously found by him. This helped us verify our findings for the collections $\ee_6$ and $\ee_7$, as described in Section~\ref{sec:real7}. We also thank Carolina Medina and anonymous referees for helpful comments.

\bibliographystyle{abbrv}
\bibliography{refs.bib}

\end{document}